\documentclass[11pt]{amsart}

\usepackage{amsmath}
\usepackage{amssymb}
\usepackage{epsfig}
\usepackage{amsfonts}
\usepackage{ytableau}
\ytableausetup{boxsize=1.2em}

\newcommand{\excise}[1]{}%$\star$\textsc{#1}$\star$}

\newtheorem{thm}{Theorem}[section]
\newtheorem{lemma}[thm]{Lemma}

\newtheorem{cor}[thm]{Corollary}

\newtheorem{conv}[thm]{Convention}

\newtheorem{ex}[thm]{Example}

\newtheorem{Warn}[thm]{Caution}

    {\end{Example}}
    {\end{Remark}}

                     {\end{minipage}\end{center}\smallskip}

\newenvironment{Abs}{\smallskip\begin{center}\begin{minipage}{13.8cm}}%
                     {\end{minipage}\end{center}\smallskip}

\numberwithin{equation}{section}

\def\wh{\widehat}

\def\sq{\square}

\def\la{\lambda}
\def\ga{\gamma}

\def\de{\delta}

\def\al{\alpha}
\def\be{\beta}

\def\cd{\mathcal D}

\def\cP{\mathcal P}

\def\ssu{\subset}

\def\<{\langle}
\def\>{\rangle}

\def\vt{\vartheta}

\def\0{{\mathbf 0}}

\def\.{\hskip.06cm}
\def\ts{\hskip.03cm}

\def\SP{{\textsc{\#P}}}

\def\pq{\delta}

\def\nin{\noindent}

\def\cD{\cd}

\def\wht{\wh}

\def\parti{\text{\small \textmd{P}}}

\def\dwn{{\hskip-.03cm \downarrow}\ts}
\def\upa{{\hskip-.03cm \uparrow}\ts}

\def\da{\vt}

\def\bos{{\langle s\rangle}}
\def\boss{{\langle 1\rangle}}

\begin{document}
\title{Bounds on  Kronecker and $q$-binomial coefficients}

\author[Igor~Pak]{ \ Igor~Pak$^\star$}

\author[Greta~Panova]{ \ Greta~Panova$^{\dagger}$}

\date{\today}

\thanks{\thinspace ${\hspace{-.54cm}}^\star$Department of Mathematics,
UCLA, Los Angeles, CA 90095; \.
\texttt{(pak@)math.ucla.edu}. \\ \thinspace ${\hspace{-.45ex}}^{\dagger}$Department
of Mathematics, University of Pennsylvania, Philadelphia, PA 19103; \. \texttt{panova@math.upenn.edu}.
}

\maketitle

\begin{Abs}{\footnotesize {\sc Abstract.} \ts
We present a lower bound on the Kronecker coefficients for tensor squares of the symmetric group via the
characters of~$S_n$, which we apply to obtain various explicit estimates.  Notably,
we extend Sylvester's unimodality of $q$-binomial coefficients $\binom{n}{k}_q$ as
polynomials in~$q$ to derive sharp bounds on the differences of their consecutive
coefficients. We then derive effective asymptotic lower bounds for a wider class of Kronecker coefficients. }
\end{Abs}

\bigskip

\section{Introduction}

\nin
The \emph{Kronecker coefficients} are perhaps the most challenging, deep
and mysterious objects in Algebraic Combinatorics.  Universally admired,
they are beautiful, unapproachable and barely understood.
For decades since they were introduced by Murnaghan in~1938, the field
lacked tools to study them, so they remained largely out of reach.
However, in recent years a flurry of activity led to significant advances,
spurred in part by the increased interest and applications to other fields.
%We refer to~\cite{PP-future} for a detailed survey of these advances and further references.

In this paper, we focus on lower bounds for the Kronecker coefficients.
We are motivated by applications to the \ts \emph{$q$-binomial $($Gaussian$\ts)$
coefficients}, and by connections to the \emph{Geometric Complexity
Theory} (see~$\S$\ref{ss:fin-GCT}).  The tools are
based on technical advances in combinatorial representation theory
obtained in recent years, see~\cite{BOR2,CDW,CHM,Man,Val2}, and our own
series of papers~\cite{PP_s,PP,PP_c,PPV}.  In fact, here we give several
extensions of our earlier work.

\smallskip

The \emph{Kronecker coefficients} \ts $g(\la,\mu,\nu)$ \ts are defined by:
\begin{equation}\label{eq:kron-def}
\chi^\la \ts \otimes \ts \chi^\mu \, = \, \sum_{\nu \vdash n}
\, g(\la,\mu,\nu)\. \chi^\nu\., \quad \text{where} \ \ \la,\mu \vdash n\ts,
\end{equation}
where $\chi^\al$ denotes the character of the irreducible representation
$\mathbb{S}^\al$
of $S_n$ indexed by partition~$\al\vdash n$.  They are integer and nonnegative
by definition, have full~$S_3$ symmetry, and satisfy a number of further
properties (see~$\S$\ref{ss:basic-kron}). In contrast with their ``cousins'',
\emph{Littlewood--Richardson $($LR$)$ coefficients},
%\ts $c^\la_{\mu\nu}$,
they lack a combinatorial interpretation or any meaningful positive formula,
and thus are harder to compute and to estimate.

\smallskip

% We apply the \emph{$k$-stable Kronecker coefficients} \ts
% $\ov{g}_k(\la,\mu,\nu)$ \ts to obtain upper bounds...

\smallskip

Our first result is a lower bound of the Kronecker coefficients $g(\la,\mu,\mu)$
for multiplicities in tensor squares of self-conjugate partitions:

\begin{thm}\label{t:char_effective}
Let $\mu=\mu'$ be a self-conjugate partition and let $\wh\mu=(2\mu_1-1,2\mu_2-3,\ldots)\vdash n$
be the partition of its principal hooks. Then:
$$
g(\la, \mu,\mu) \, \geq \, \bigl|\ts\chi^\la[\wh\mu] \ts \bigr|\,, \quad \text{for every}  \quad \la \vdash n\ts,
$$
where $\chi^\la[\wh \mu]$ denotes the value of the character $\chi^\la$ at a permutation of cycle type $\wh\mu$.
\end{thm}

While it is relatively easy to obtain various upper bounds on the Kronecker
coefficients (see e.g.~\eqref{eq:upper-cat}), this is the only general lower bound that we know.
The theorem strengthens a qualitative result \ts $g(\la, \mu,\mu) \ge 1$ \ts given
in~\cite[Lemma~1.3]{PPV}, used there to prove a special case of the Saxl conjecture
(see $\S$\ref{ss:fin-char}).  We use the bound to give a new proof of
Stanley's Theorem~\ref{t:stanley-unimod}, from~\cite{Sta-unim}.%; the only other proof we know is the original proof based on the Weak Lefschetz
% Property~\cite{Sta-unim}.
% (see~$\S$\ref{ss:fin-unimod}).
%
% Other applications of...

\smallskip

Our second result is motivated by an application of bounds for Kronecker
coefficients to the \emph{$q$-binomial coefficients}, defined as:
$$
\binom{m+\ell}{m}_q
\, = \ \. \frac{(q^{m+1}-1)\. \cdots\. (q^{m+\ell}-1)}{(q-1)\.\cdots\. (q^{\ell}-1)}
\ \. = \, \, \sum_{n=0}^{\ell\ts m} \, \. p_n(\ell,m) \. q^n\ts,
$$
where $p_n(\ell,m)$ is also the number of partitions  of $n$ which fit inside an $\ell \times m$ rectangle.
In~1878, \emph{Sylvester} proved \emph{unimodality} of the coefficients:
$$
% (\star )\quad
p_0(\ell,m)\. \le \. p_1(\ell,m)\. \le \. \ldots \. \le \. p_{\lfloor\ell\ts m/2\rfloor}(\ell,m) \. \ge \. \ldots \. \ge \. p_{\ell\ts
m}(\ell,m)\ts,
$$
see~\cite{Syl}.  In~\cite{PP_s}, we used the Kronecker coefficients to prove \emph{strict unimodality}:
\begin{equation}\label{eq:strict-unim}
p_k(\ell,m) \ts - \ts p_{k-1}(\ell,m) \. \ge 1\., \
\quad \text{for} \ \quad 2\le k \le \ell\ts m/2\ts, \ \. \ell, \ts m \ge 8\ts.
\end{equation}
This result was subsequently improved by Zanello~\cite{Zan} and Dhand~\cite{Dha}
(see~$\S$\ref{ss:fin-qbin}).
Ignoring constraints of $\ell$ and~$m$, they prove that the l.h.s. of~\eqref{eq:strict-unim}
is $\Omega(\sqrt{k})$ and $\Omega(k)$, respectively.  Here we substantially strengthen
these bounds as follows and give an effective bound on certain Kronecker coefficients.

\begin{thm}\label{t:qbin-main}
There is a universal constant $A>0$, such that for all $m\ge \ell\ge 8$ and
$2\le k\le \ell \ts m/2$, we have:
$$
g((\ell m -k, k), m^\ell, m^\ell) = p_k(\ell,m) \ts - \ts p_{k-1}(\ell,m) \, \ts > \, A \, \frac{2^{\sqrt{s}}}{s^{9/4}}\,, \quad \ \text{where} \quad \.
s=\min\{2k,\ell^2\}\ts.
$$
\end{thm}

The proof of the theorem gives an effective bound with $A= 0.004$.
The lower bound gives the correct exponential behavior of the difference,
but perhaps not the base of the exponent.
We also discuss an upper bound in~$\S$\ref{ss:qbin-upper}
(see also~$\S$\ref{ss:fin-qbin}).

\smallskip

The proof of the theorem has several ingredients.
We use the above mentioned Stanley's theorem, an extension of analytic
estimates in the proof of \emph{Almkvist's Theorem}
(Theorem~\ref{t:almkvist-unimod}), and the \emph{monotonicity property}
of the Kronecker coefficients (Theorem~\ref{t:manivel}).
Most crucially, we use the following connection between
the Kronecker and $q$-binomial coefficients:

\begin{lemma}[Two Coefficients Lemma]
\label{l:g_partitions}
Let $n=\ell \ts m$, \ts $\tau_k=(n-k,k)$, where $1\leq k\leq n/2$.
Then:
$$g\bigl(m^\ell,m^\ell,\tau_k\bigr) \, = \, p_k(\ell,m) \. - \. p_{k-1}(\ell,m)\ts.
$$
\end{lemma}

This simple but very useful lemma was first proved in a special case
in~\cite{CM}, and in full generality in~\cite[$\S 7$]{Val2}
and later in~\cite{PP_s}, but is implicit in~\cite{MY,PP}.  Note that it
immediately implies Sylvester's unimodality theorem.

\smallskip

Finally, using this result and the semigroup property for
Kronecker coefficients we can give an explicit lower bound
for a wider classes of partition triples.  Here we state
an easy corollary of a general (but technical to state)
Theorem~\ref{thm:gen} (see below).

\begin{cor}\label{cor:general}
For any partition $\la \vdash n$, let $d(\la) =m$ be the
size of its Durfee square. Then, for any $k \leq m^2/2$, we have:
$$
g(\la, \la, (n-k,k) ) > C  \frac{2^{\sqrt{2k} } }{(2k)^{9/4} }\,,
\quad \text{where} \quad C\ts = \. \frac{\sqrt{27/8}}{\pi^{3/2}}\..
$$
\end{cor}

Here the \emph{Durfee square} is the largest square which fits into
Young diagram of the partition. In other words, the ``thicker''
 the partition is, the better lower bound we obtain.

\smallskip

The rest of the paper is structured as follows.  We begin with a quick recap
of definitions, notations and some basic results we are using (Section~\ref{s:basic}).
We prove Theorem~\ref{t:char_effective} in Section~\ref{s:effective-proof} and
give applications of the theorem in Section~\ref{s:char}. We then prove
Theorems~\ref{t:qbin-main}   in Section~\ref{s:qbin} and the more general effective bound on the Kronecker coefficients in Section~\ref{sec:gen} as Theorem~\ref{thm:gen}.  We conclude with final
remarks and open problems (Section~\ref{s:fin}).

\bigskip

\section{Definitions and basic results}\label{s:basic}

\subsection{Partitions and Young diagrams}\label{ss:basic-part}
We adopt the standard notation in combinatorics of partitions
and representation theory of~$S_n$, as well as the theory of symmetric functions
(see e.g.~\cite{Mac,Sta}).

Let $\cP$ denote the set of integer partitions $\la = (\la_1,\la_2,\ldots)$.  We write
$|\la|=n$ and $\la\vdash n$, for $\la_1+\la_2+\ldots = n$. Let $\cP_n$ the set of
all $\la\vdash n$, and let $\parti(n)=|\cP_n|$ the number of partitions of~$n$.
We use $\ell(\la)$ to denote the number of parts of~$\la$, and $\la'$ to denote
the conjugate partition.
Define addition of partitions $\al,\be \in \cP$ to be their addition as vectors:
$$\alpha+\beta \. = \. (\alpha_1+\beta_1, \ts \alpha_2+\beta_2,\ts \ldots)\ts.$$

We denote by $\chi^\la$ the character of the irreducible representation $\mathbb{S}^\la$
of $S_n$ corresponding to~$\la$. Denote by  $f^{\la}=\chi^\la[1^n]$ the dimension of
$\mathbb{S}^\la$.  Finally, \emph{hooks} of a partition  $\mu$ are defined by
$h_{ij}=\mu_i+\mu_j'-i-j+1$, and the integers $h_{11},h_{22},\ldots$
are called \emph{principal hooks}. When $\mu=\mu'$, the sequence of principal hooks
is exactly the partition $\wh\mu$ defined in Theorem~\ref{t:char_effective}.

\subsection{Kronecker coefficients}\label{ss:basic-kron}
It is well known that
$$
g(\la,\mu,\nu) \, = \, \frac{1}{n!} \. \sum_{\omega \in S_n} \. \chi^\la(\omega)\ts \chi^\mu(\omega)\ts \chi^\nu(\omega)\ts.
$$
This implies that Kronecker coefficients have full $S_3$ group of symmetry:
$$
g(\la,\mu,\nu) \, = \, g(\mu,\la,\nu) \, = \, g(\la,\nu,\mu) \, = \, \ldots % \,,
$$

We will use the following \emph{monotonicity property}:

\begin{thm}[\cite{Man}]  \label{t:manivel}
Suppose $\al,\be,\ga$ are partitions of~$n$, such that the Kronecker coefficients
$ g(\al,\be,\ga)> 0$. Then for any partitions $\la,\mu,\nu$ with $|\la|=|\mu|=|\nu|$ we have
$$g(\la+\al,\mu+\be,\nu+\ga)\. \geq \. g(\la,\mu,\nu)\ts .$$
\end{thm}

This result is an extension of the \emph{semigroup property} for the Kronecker coefficients,
which states that the LHS~$>0$ if $g(\la,\mu,\nu)>0$.  This property was
originally attributed to Brion and reproved in~\cite{CHM,Man}.

\smallskip

We also have the following trivial upper bound (see e.g.~\cite[Exc.~7.83]{Sta})~:
\begin{equation}\label{eq:upper-cat}
g(\la,\mu,\nu) \leq \min\{f^\la,f^\mu,f^\nu\} \quad \text{for all} \quad \la,\mu,\nu\vdash n\ts.
\end{equation}

\smallskip

\subsection{Partition asymptotics}\label{ss:basic-asympt}
Denote by \ts $\parti'(n) = \parti(n) - \parti(n-1)$ the number of partitions into
parts~$\ge 2$. We have that $\parti'(n)\ge 1$ for all $n\ge 2$. Recall the following
\emph{Hardy--Ramanujan} and \emph{Roth--Szekeres} formulas,
respectively:
\begin{equation} \label{HR asymptotics}
 \parti(n) \. \sim \. \frac{1}{4\sqrt{3}\ts n} \,\, e^{\pi \sqrt{\frac{2}{3}\ts n}} \,,\quad
  \parti'(n) \. \sim \. \frac{\pi}{\sqrt{6\ts n}} \,\parti(n)  \, \ \quad \text{as } \ n\to\infty\.,
 \end{equation}
see~\cite{RS} (see also~\cite[p.~59]{ER}).

Denote by $b_k(n)$ the number of partitions of $k$ into distinct odd parts~$\le 2n-1$.  We have:
$$
\prod_{i=1}^{n}\. \ts \bigl(1+q^{2i-1}\bigr)\ \. = \, \, \sum_{k=0}^{n^2} \, \. b_k(n) \. q^k\ts.
$$
\begin{thm}[Almkvist] \label{t:almkvist-unimod}
The following sequence is symmetric and unimodal for $n>26$:
$$
(\lozenge)\quad
b_{2}(n)\ts, \, b_{3}(n)\ts, \, \ldots \,, \, b_{n^2-2}(n)\ts.
$$
\end{thm}

\bigskip

\section{Proof of Theorem~\ref{t:char_effective}}\label{s:effective-proof}

Denote by $\chi\dwn$ the restriction of the $S_n$-representation $\chi$
to~$A_n$, and by $\psi\upa$ the induced $S_n$-representation of the
$A_n$-representation~$\psi$.  We refer to~\cite[$\S$2.5]{JK} for basic results in representation theory of~$A_n$.
Recall that if $\nu\neq \nu'$, then $\chi^\nu\dwn = \chi^{\nu'}\dwn=\psi^\nu$ is irreducible in~$A_n$.
Similarly, if $\nu=\nu'$, then $\chi^\nu\dwn = \psi^\nu_+ \oplus \psi^\nu_-$, where $\psi^\nu_{\pm}$
are irreducible in~$A_n$, and are related via $\psi^\nu_+[ (12)\pi(12)] = \psi^\nu_-[\pi]$.

Consider now the conjugacy classes of~$A_n$ and the corresponding character values.
Denote by $C^{\alpha}$ the conjugacy class of $S_n$ of permutations of cycle type~$\alpha$, and
by $\cD\ssu \cP$ the set of partitions into distinct odd parts.
We have two cases:

\smallskip

\nin
{\small $\mathbf{(1)}$} \ For $\alpha\notin\cD$, we have $C^\alpha$ is also a conjugacy class of~$A_n$.  Then
$$
\aligned
\chi^\nu\dwn[C^{\alpha}] \,  = \, \chi^\nu[C^\alpha] \quad & \text{if} \ \ \nu \neq \nu'\ts,\\
\psi^\nu_{\pm}[C^\alpha] \,  = \, \frac12 \. \chi^\nu[C^\alpha]\quad & \text{if} \ \ \nu=\nu'\ts.
\endaligned
$$

\nin
{\small $\mathbf{(2)}$} \ For $\alpha\in\cD$, we have $C^\alpha = C^{\alpha}_+ \cup C^{\alpha}_-$, where $C^{\alpha}_{\pm}$ are conjugacy classes
of~$A_n$. Then
$$
\aligned
\chi^\nu \dwn [C^{\alpha}_{\pm}] \,  = \, \chi^\nu [C^\alpha] \quad & \text{if} \ \ \nu \neq \nu'\ts,\\
\psi^{\nu}_{\pm} [C^\alpha_{\pm}] \, = \, \frac12 \chi^\nu[C^\alpha] \quad & \text{if} \ \ \nu = \nu' \ \ \. \text{and} \ \ \. \alpha \neq
\wh\nu\ts, \\
\psi^{\nu}_{\pm}[C^{\wht{\nu}}_+] - \psi^{\nu}_{\pm}[C^{\wht{\nu}}_-] \,  = \, \pm e_{\nu}  \quad & \text{if} \ \ \nu=\nu' \ \ \. \text{and} \ \
\.
e_{\nu} = (\wh\nu_1\ts \wh\nu_2 \ts \cdots\ts)^{1/2}>0\ts.
\endaligned
$$

\nin
Now, by the Frobenius reciprocity, for every \ts $\mu=\mu'$ \ts we have:
$$
\langle \psi^{\mu}_{\pm}\upa , \chi^{\alpha}\rangle =  \langle \psi^{\mu}_\pm , \chi^\alpha\dwn\rangle\ts,
$$
which is nonzero exactly when $\alpha=\mu$ and so $\psi^{\mu}_{\pm}\upa = \chi^\mu$.  This implies
\begin{equation}\label{eq:kron}
\aligned
g(\la,\mu,\mu) \, & = \, \langle \chi^\mu\otimes \chi^\la, \chi^\mu \rangle \, = \,
\langle \chi^\mu\otimes \chi^\la, \ts \psi^\mu_{\pm}\upa \rangle \, = \,
\bigl\langle (\chi^\mu\otimes \chi^\la)\dwn, \psi^\mu_{\pm} \bigr\rangle \\
& = \,
\bigl\langle \psi^\mu_+ \ts \otimes \ts \chi^\la\dwn, \ts \psi^\mu_{\pm} \bigr\rangle \. + \.
\bigl\langle \psi^\mu_-\ts \otimes \ts \chi^\la\dwn, \ts \psi^\mu_{\pm} \bigr\rangle\ts.
\endaligned
\end{equation}

\smallskip
We can now estimate the Kronecker coefficient in the theorem.
First, decompose the following tensor product of the $A_n$ representations:
\begin{equation}\label{eq:decomp}
 \psi^{\mu}_+ \ts \otimes \ts  \chi^\la\dwn \, = \, \oplus_{\tau} \. m_\tau \ts \psi^\tau\ts,
 \end{equation}
where $\psi^\tau$ are all the irreducible representations of~$A_n$, the coefficients~$m_\tau$
are their multiplicities in the above tensor product, and $\tau$ goes over the appropriate indexing.

Note that for any character $\chi$ of $S_n$ and $\pi \in A_n$ we trivially have $\chi\dwn [\pi] =\chi[\pi]$.
% , where $\pi$ is any permutation from the given conjugacy class.
Evaluating that tensor product on the classes $C^{\wht\mu}_{\pm}$ gives
$$
\bigl(\psi^\mu_+\otimes \chi^\la\dwn\bigr)\bigl[ C^{\wht{\mu}}_+\bigr] \. - \. \bigl(\psi^\mu_+\otimes \chi^\la\dwn\bigr)\bigl[
C^{\wht{\mu}}_-\bigr] \,
= \,
\chi^\la\dwn\bigl[C^{\wht{\mu}}_{\pm}\bigr] \ts \Bigl( \psi^\mu_+\bigl[C^{\wht\mu}_+\bigr] \. - \. \psi^\mu_+\bigl[C^{\wht\mu}_-\bigr] \Bigr) \,
=\, \chi^\la\bigl[C^{\wht{\mu}}\bigr] \ts e_{\nu} \..
$$
On the other hand, evaluating the right-hand side of equation~\eqref{eq:decomp} gives
$$\aligned
 &\qquad \quad  \bigl(\psi^\mu_+\otimes \chi^\la\dwn \bigr)\bigl[ C^{\wht{\mu}}_+\bigr] \. - \. \bigl(\psi^\mu_+\otimes \chi^\la\dwn\bigr)\bigl[
 C^{\wht{\mu}}_-\bigr]
 \, = \,
\sum_\tau \. m_\tau \Bigl(\psi^\tau\bigl[C^{\wht\mu}_+\bigr] - \psi^\tau\bigl[C^{\wht\mu}_-\bigr]\Bigr) \\
& = \, m_{\mu+}\Bigl(\psi^{\nu}_{+}\bigl[C^{\wht{\nu}}_+\bigr] - \psi^{\nu}_{+}\bigl[C^{\wht{\nu}}_-\bigr]\Bigr)\.
+ \.m_{\mu-}\Bigl(\psi^{\nu}_{-}\bigl[C^{\wht{\nu}}_+\bigr] - \psi^{\nu}_{-}\bigl[C^{\wht{\nu}}_-\bigr]\Bigr) \,
= \, \bigl(m_{\mu+}-m_{\mu-}\bigr) \ts e_{\nu}.
\endaligned
$$
Here we used the fact that all characters are equal at the two classes $C^{\wht{\mu}}_{\pm}$, except for the ones corresponding to~$\mu$.
Equating the evaluations and using $e_\nu>0$, we obtain
$$
m_{\mu+}\ts - \ts m_{\mu-} \. = \. \chi^\la\bigl[C^{\wht{\mu}}\bigr]\ts.
$$
This immediately implies
\begin{equation}\label{eq:max-char}
\max \bigl\{m_{\mu+},m_{\mu-}\bigr\} \, \geq \, \Bigl| \chi^\la\bigl[C^{\wht{\mu}}\bigr] \Bigr|
\end{equation}

On the other hand, since all inner products are nonnegative, the equation~\eqref{eq:kron} gives
$$
g(\la,\mu,\mu) \, \geq \,
\max\left\{\langle \psi^\mu_+\otimes \chi^\la\dwn, \. \psi^\mu_{+} \rangle,
\langle \psi^\mu_+\otimes \chi^\la\dwn, \psi^\mu_{-} \rangle\right\} \, =\,
\max \bigl\{m_{\mu+},m_{\mu-}\bigr\}\.,
$$
and now equation~\eqref{eq:max-char} implies the result. \ $\sq$

\bigskip

\section{Bounds on Kronecker coefficients via characters}\label{s:char}

\subsection{Stanley's theorem} \label{ss:stanley}
We give a new proof of the following technical result
by Stanley \cite[Prop.~11]{Sta-unim}.  Our proof uses Theorem~\ref{t:char_effective}
and Almkvist's Theorem~\ref{t:almkvist-unimod}.  Both results are
crucially used in the next section.

\begin{thm}[Stanley] \label{t:stanley-unimod}
The following polynomial in $q$ is symmetric and unimodal
$$
\binom{2n}{n}_q \, - \, \. \prod_{i=1}^{n}\. \ts \bigl(1+q^{2i-1}\bigr)\..
$$
\end{thm}

\begin{proof}
Let $\mu=(n^n)$ and $\tau_k=(n^2-k,k)$, where $k\le n^2/2$.
By the two coefficients lemma (Lemma~\ref{l:g_partitions}),
we have
$$g(\tau_k,\mu,\mu) \, = \, p_k(n,n) \. - \. p_{k-1}(n,n)\ts.
$$
By the Jacobi-Trudi identity %Frobenius character formula  --- this is something else according to the literature I find (Etingof), so let's not
% use that name, ok?
and the  Murnaghan--Nakayama rule, we have:
$$\chi^{\tau_k}[\wht{\mu}] \, = \, \chi^{(n^2-k) \circ (k) }\bigl[\wht{\mu}\bigr] \.
- \. \chi^{(n^2-k+1) \circ (k-1)}\bigl[\wht{\mu}\bigr] \,
=\, b_k(n) - b_{k-1}(n)\ts.
$$
(cf.~\cite{PP,PPV}).
Applying Theorem~\ref{t:char_effective} with $\la=\tau_k$ and $\mu$ as above, we have:
$$p_k(n,n)\. -\. p_{k-1}(n,n) \, = \, g(\la,\mu,\mu) \,
\geq \,\chi^{\la}\bigl[\wht{\mu}\bigr]  \, = \, b_k(n)\. - \. b_{k-1}(n)\ts.
$$
%The last equality follows from Almkvist's Theorem~\ref{t:almkvist-unimod}.
Reordering the terms, we conclude
$$p_k(n,n) \. - \. b_{k}(n) \, \geq \, p_{k-1}(n,n) - b_{k-1}(n)\ts,
$$
which implies unimodality.  The symmetry is straightforward.
\end{proof}

\subsection{Asymptotic applications}
Let $\rho_m=(m,m-1,\ldots,2,1)$ be the
\emph{staircase shape}, $n=|\rho_m|=\binom{m+1}{2}$. The coefficient
$g(\rho_m,\rho_m,\nu)$ first appeared in connection with the
\emph{Saxl conjecture}~\cite{PPV}, and was further studied
in~\cite[$\S$8]{Val2}.

For simplicity, let $m=1$~mod~2, so
\ts $\wh \rho_m = (2m-1,\ldots,5,1)$.  Let $\tau_k=(n-k,k)$.
Applying Theorem~\ref{t:char_effective} and
the Murnaghan--Nakayama rule as above,
we have
$$g\bigl(\rho_m,\rho_m,\tau_k) \, \ge \,
\bigl|\chi^{\tau_k}\bigl[\wh \rho_m\bigr]\ts\bigr| \, = \, \parti_R(k) \. - \.
\parti_R(k-1)\ts,$$
where $\parti_R(k)$ is the number of partitions of~$k$ into parts from
$R=\{1,5,\ldots,2m-1\}$.

In the ``small case'' \ts $k \le 2m$, by the Roth--Szekeres theorem~\cite{RS},
we have:
$$g\bigl(\rho_m,\rho_m,\tau_k) \, \ge \, \parti_R(k) \. - \.
\parti_R(k-1) \, \sim \, \frac{\pi\ts\sqrt{2}}{3\ts k^{3/2}} \.
e^{\pi\sqrt{k/6}}\.,
$$
i.e.~independent of~$n$.
On the other hand, by equation~\eqref{eq:upper-cat}, we have
$$
g\bigl(\rho_m,\rho_m,\tau_k) \, \le \, f^{\tau_k} \, < \, \frac{n^k}{k!}\ts,
$$
leaving a substantial gap between the upper and lower bounds.
For $k=O(1)$ bounded, Theorem~8.10 in~\cite{Val2}, gives
$$
g\bigl(\rho_m,\rho_m,\tau_k) \, \sim \, m^k \, \sim \, (2\ts n)^{k/2}\quad
\text{as} \ \
n\to \infty \ts,
$$
suggesting that the upper bound is closer to the truth.  In fact, the proof
in~\cite{Val2} seems to hold for all $k=o(m)$.

In the ``large case'' \ts $k =n/2\sim m^2/4$, the Odlyzko--Richmond result (\cite[Thm.~3]{OR}) gives
$$g\bigl(\rho_m,\rho_m,\tau_k) \, \ge \,
\parti_R(k) \. - \.
\parti_R(k-1) \, \sim \, \frac{3^{3/2}}{2^{15/4}\ts \sqrt{\pi} \ts m^3} \. 2^{m/4}
\, \sim \, \frac{3^{3/2}}{2^{47/4}\ts \sqrt{\pi} \ts k^{3/2}} \. \. 2^{\sqrt{k}/2}
\ts.
$$
For the upper bound, equation~\eqref{eq:upper-cat} gives
$$g\bigl(\rho_m,\rho_m,\tau_k) \, \le \, f^{\tau_k} \, \lesssim \, \frac{1}{\sqrt{\pi} \ts k^{3/2}} \, 4^k\ts.
$$

\subsection{Lower bounds for border equal partitions} \label{ss:border}
Two partitions $\la,\mu\vdash n$ are called \emph{$s$-border equal} if
they have equal the first $s$ principal hooks.  By $\la^\bos$ denote the partition
with the first $s$~rows and $s$~columns removed.

\begin{cor} \label{cor:border}
Let $\la, \mu \vdash n$ be $s$-border equal partitions such that
$\mu=\mu'$ is self-conjugate.  Denote by $\al=\la^\bos$, $\be = \mu^\bos$,
and let $\wh\be=(2\be_1-1,2\be_2-3,\ldots)\vdash n$. Then:
$$
g(\la, \mu,\mu) \, \geq \, \bigl|\ts\chi^\al[\wh\be] \ts \bigr|\..
$$
\end{cor}

\begin{proof}
By the Murnaghan--Nakayama rule, for the $s$-border equal partitions $\la$ and~$\mu$,
there is a unique way to fit the first~$s$ rim hooks of length
$(\wh\mu_1,\ldots,\wh\mu_s)$ into the shape~$\la$.  Therefore, we have
$$\bigl|\ts\chi^\la[\wh\mu] \ts \bigr| \, = \, \bigl|\ts\chi^\al[\wh\be] \ts \bigr|\..
$$
Now Theorem~\ref{t:char_effective}  implies the result.
\end{proof}

\begin{ex} {\rm
Fix $r\ge 4$ and $s\ge 0$.  Consider $\la= \bigl(2r^2+2r+s+1\bigr)^{s+1}\bigl(s+1\bigr)^{2r^2+2r}$,
$\mu = \bigl(2r^2+2r+s+1\bigr)^s(2r+s+1)^{2r+1}s^{2r^2}$, and observe that $|\la|=|\mu|$,
$\mu = \mu'$.  Furthermore, $\la$ and $\mu$ are $s$-border equal. In notation of the corollary, we have $\al = \bigl(2r^2+2r+1,1^{2r^2+2r}\bigr)$,
$\be=\bigl(2r+1\bigr)^{2r+1}$, and $\wh \be = (4r+1,4r-1,\ldots,3,1)$.  Using
Corollary~\ref{cor:border}, the Murnaghan--Nakayama rule and the Giambelli formula as in~\cite{PPV},
we conclude that
$$
g(\la,\mu,\mu) \, \ge \,\, \bigl|\ts\chi^\al[\wh\be] \ts \bigr| \, \ge \, b_{k}(2r+1) \ts - \ts b_{k-1}(2r+1)\ts.
$$
 We therefore have:
\begin{align}\label{eq:dimo} \qquad \
g(\la,\mu,\mu) \, \ge \, 0.32 \ts\frac{2^{\sqrt{2k} }}{(2k)^{9/4}} \, \geq \,
\frac{4^r} {3\ts (2r)^{9/2} }\,\,, \,\, \quad \text{where} \quad k \ts =\ts 2\ts r\ts (r+1)\ts.
\end{align}
where the inequality~\eqref{eq:dimo} follows from Theorem~\ref{t:almkvist_eff} given in the next section (note that $r\ge 4$ implies $k \geq 26$,
assumed in the theorem).  We omit the details.
}\end{ex}

\bigskip

\section{Bounds on the $q$-binomial coefficients}\label{s:qbin}

\subsection{Analytic estimates} \label{ss:qbin-alm}
The proof of Almkvist's Theorem~\ref{t:almkvist-unimod} is based on the
following technical results.

\begin{lemma}[\cite{A1}]  \label{l:alm1}
For $3\leq k \leq 2n+1$, we have:
$$
b_k(n)\ts - \ts b_{k-1}(n) \, = \, \begin{cases} b_k(n-1) \ts - \ts b_{k-1}(n-1)\ts, & k\neq 2n\pm 1,\\
b_k(n-1)-b_{k-1}(n-1)+1 \ts, & k=2n-1 \\
b_k(n-1)-b_{k-1}(n-1)-1 \ts, & k=2n+1\\
\end{cases}
$$
Similarly, for $2n+2\leq k\leq (n-1)^2/2$, we have:
$$
b_k(n)\ts -\ts b_{k-1}(n) \, = \, b_k(n-1)\ts -\ts b_{k-1}(n-1) \ts +\ts b_{k-2n+1}(n-1)\ts -\ts b_{k-2n}(n-1)\ts.
$$
\end{lemma}

\smallskip

\begin{lemma}[\cite{A1}]  \label{l:alm2}
For $n\geq 86$ and $(n-1)^2/2\leq k \leq n^2/2$, we have:
$$
b_k(n) \ts -\ts b_{k-1}(n) \, \geq \, C \. \frac{2^n}{n^{9/2} }\,,
\quad \text{where} \quad C\ts = \. \frac{3\ts\sqrt{3}}{2 \sqrt{2} \ts \pi^{3/2}} \. \approx \ts 0.329\..
$$
\end{lemma}
\begin{proof}
We invoke details from the computations in~\cite{A1}. It is shown there that 
$$ \frac{ \partial b_k(n)}{\partial k} \geq \frac{2^{n+2}}{\pi} I \geq \frac{2^{n+3}}{\pi} ( I_1 - |I_2| - |I_3| - |I_4|),$$
where $I_j$ are certain explicit integrals. By the bounds in this proof we have:
\begin{align}\label{eq:I_ineq}
I_1 \geq \frac{3 \sqrt{3}}{4\sqrt{2 \pi} } \frac{1}{n^{9/2} }, \qquad |I_2| \leq \frac{7 \pi^3}{ 96 n^2}\exp\left(- \frac{(5\pi-6)n}{16\pi} + \frac{\pi}{32 n} \right), \\
|I_3| \leq \frac{\pi^2}{8}\exp\left( - \frac{(5\pi -3)n}{16 \pi} +\frac{\pi}{16 n} \right), \qquad |I_4| \leq \frac{\pi \sqrt{\pi n} }{4} \left( \frac{3}{\sqrt{2 e} n} \right)^n\end{align}
A calculation shows that for $n\geq 86$ we have
\begin{align*}
\frac{1}{2} \frac{3 \sqrt{3}}{4\sqrt{2 \pi} } \frac{1}{n^{9/2} }&\geq \frac{7 \pi^3}{ 96 n^2}\exp\left(- \frac{(5\pi-6)n}{16\pi} - \frac{\pi}{32 n} \right) \\ &+\frac{\pi^2}{8}\exp\left( - \frac{(5\pi -3)n}{16 \pi} +\frac{\pi}{16 n} \right) + \frac{\pi \sqrt{\pi n} }{4} \left( \frac{3}{\sqrt{2 e} n} \right)^n . \end{align*}
Applying this to the inequality~\eqref{eq:I_ineq} we get
$$ b_k(n) - b_{k-1}(n) \geq \frac{\partial b_k(n)}{\partial k}  \geq \frac{2^{n+2}}{\pi} ( I - |I_2| -|I_3| - |I_4|) \geq  \frac{2^{n+1}}{\pi} \frac{3 \sqrt{3}}{4\sqrt{2 \pi} } \frac{1}{n^{9/2} }.$$
\end{proof}

Based on this setup, we refine Almkvist's Theorem~\ref{t:almkvist-unimod} as follows

\begin{thm}\label{t:almkvist_eff}
For any $n\geq 31$, and $26 \leq k \leq  n^2/2$ we have:
$$
b_k(n)\. -\. b_{k-1}(n) \, \geq \, C \.  2^{\sqrt{2\ts k}} \ts \frac{ 1}{(2\ts k)^{9/4} }\.,
\quad \text{where \ts $C$ \ts is as above.}
$$
\end{thm}

\begin{proof}
Denote
$$\da_k(n) \, = \, b_k(n)\. -\. b_{k-1}(n)\ts.
$$
First, let \ts $n\geq 83$ and \ts $(n-1)^2/2 \leq k \leq n^2/2$.  By Lemma~\ref{l:alm2}, we have
\begin{equation}\label{eq:aml-mid}
\da_k(n) \, \geq \, C \. 2^n \. \frac{ 1 }{n^{9/2} } \, \geq \, C \. 2^{\sqrt{2\ts k}} \. \frac{ 1}{(2k)^{9/4} }\,,
\end{equation}
where the last inequality follows since the function \ts $f(x)=\log2 \sqrt{x} -9/4 \log (x)$ \ts
is increasing.

The recurrence relations in Lemma~\ref{l:alm1} and Almkvist's Theorem $\vt_k(n)\ge 0$ give
$$
\da_k(n) \, \geq \, \da_k(n-1) \quad \text{for all} \quad  3 \le k \le n^2/2 \ts , \; k\neq 2n+1 \ts
$$
and $\da_{2n+1}(n) \, = \, \da_{2n+1}(n-1)-1 \, =\, \da_{2n+1}(n-4)$.
Now, let $r$ be such that $(n-r-1)^2/2 \leq k \leq (n-r)^2/2$, and $n-r \geq 83$.
Applying~\eqref{eq:aml-mid} to $(n-r)$, we conclude:
$$\da_k(n)\, \geq \, \da_k(n-r) \, \geq \, C \. 2^{\sqrt{2k}} \. \frac{ 1}{(2k)^{9/4} }\,.
$$
Next, we check by computer that the inequality in the theorem holds for all $n \in \{31,\ldots,86\}$ and $26 \leq k \leq n^2/2$. Finally, for
$k\leq 86^2/2$ and $n > 85$, we apply the inequalities of Lemma~\ref{l:alm1} repeatedly to obtain
$$\da_k(n)\, \geq \, \da_k(86) \, \geq \, C \. 2^{\sqrt{2k}} \. \frac{ 1}{(2k)^{9/4} }\,. \hfill \qedhere$$
\end{proof}

\begin{cor}\label{c:square_bin}
Let $n \geq 8$, $1 \le k \le n^2/2$, $\mu=(n^n)$ and $\tau_k=(n^2-k, k)$.  Then
$$
g\bigl(\mu,\mu,\tau_k \bigr)\, \geq \, C \. \frac{ 2^{\sqrt{2\ts k}} }{(2\ts k)^{9/4} }\,,
\quad \text{where} \quad C\ts = \. \frac{\sqrt{27/8}}{\pi^{3/2}}\..
$$
\end{cor}
\begin{proof}
Following the proof of Stanley's Theorem~\ref{t:stanley-unimod}, for all $26\leq k\leq n^2/2$ and $n\geq 31$
Theorem~\ref{t:almkvist_eff} gives:
$$
g\bigl(\mu,\mu,\tau_k \bigr)\, = \, p_k(n,n) - p_{k-1}(n,n) \,\geq
\, b_k(n)-b_{k-1}(n) \, \geq \, C \. \frac{ 2^{\sqrt{2\ts k}} }{(2\ts k)^{9/4} }\. .
$$
For the remaining values of $n$ and $k$ we check the inequality by computer, noticing that $p_k(n,n) =p_k(26,26)$ when $k\leq 26$.
\end{proof}

%\bigskip

\subsection{Partitions in rectangles} \label{ss:qbin-rect}
%
% Denote \ts $\pq_k(\ell,m) \ts = \ts p_k(\ell,m) -  p_{k-1}(\ell,m)$.
Let $\tau_k=(m\ell -k,k)$, by Lemma~\ref{l:g_partitions}, we have
$$\pq_k(\ell,m) \, := \, p_k(\ell,m) \. - \. p_{k-1}(\ell,m) \, = \, g(m^\ell,m^\ell,\tau_k)\ts.
$$

\begin{thm}\label{t:qbin-rect}
Let $8\le \ell \le m$ and $1\le k \le m\ell/2$.
Define $n$ as $$n = \begin{cases} 2\lfloor \frac{\ell-8}{2} \rfloor ,& \text{ when $\ell m$ is even, }\\
2 \lfloor \frac{\ell-8}{2}\rfloor -1, & \text{ when $\ell m$ is odd, }
\end{cases}$$
and let $v=\min(k, n^2/2)$. Then:
$$
\pq_k(\ell,m) \, \geq \, C \frac{ 2^{\sqrt{2v}} }{(2v)^{9/4} } \qquad
\text{where} \quad C \. = \. \frac{3\sqrt{3}}{2\sqrt{2}\.\pi^{3/2}}\..
$$
\end{thm}

\begin{proof}
We apply Theorem~\ref{t:manivel} to bound the Kronecker coefficient for rectangles with an appropriate Kronecker coefficient for a square and then
apply Corollary~\ref{c:square_bin}.

By strict unimodality~\eqref{eq:strict-unim}, we have that $g(m^\ell,m^\ell, (m\ell-k,k) )>0$ for all $\ell,m \geq 8$.
By Corollary~\ref{c:square_bin}, we can assume $\ell <m$. Assume first that $\ell >16$.

First, suppose that $\ell m$ is even and let $n= 2\lfloor \frac{\ell-8}{2} \rfloor $. Then for any $1<k \leq \frac{\ell m}{2}$ we can find $1\neq
r\leq \frac{ (m-n)\ell}{2}$ and $1\neq s \leq \frac{ n\ell}{2} $, such that $k=r+s$. Take $s=\min(k,n\ell/2)$ . Let $\tau_k = (m\ell-k,k)$ and
$\tau_r = \left( (m-n)\ell-r,r \right)$, $\tau_s= \left( n\ell-s,s\right)$. Apply Theorem~\ref{t:manivel} to the triples $\left( (m-n)^\ell,
(m-n)^\ell, \tau_r \right)$ and $\left( n^\ell, n^\ell, \tau_s\right)$ to obtain

$$\pq_k(\ell,m) \, = g( m^\ell, m^\ell, \tau_k) \geq \max\left( g\bigl( (m-n)^\ell, (m-n)^\ell, \tau_r \bigr), g\bigl( n^\ell, n^\ell,
\tau_s\bigr) \right)\, \geq \, \pq_s(\ell,n).$$
Similarly, dividing the $n \times \ell$ rectangle into $n\times n$ square and $n \times (n-\ell)$ rectangle, where $n\ell$ is even, we have
$$\pq_s(\ell,n) \geq \pq_{v}(n,n),$$
where $v=\min(s, n^2/2)=\min(k, n^2/2)$.

In the case that both $\ell$ and $m$ are odd, the only case where the above reasoning fails is when $k=\lfloor m\ell /2 \rfloor$ and $r,s$ don't
exist. In this case we take $n=2 \lfloor \frac{\ell-8}{2}\rfloor -1$ and we can always find $r,s$. In summary, we have that
$$\pq_k(\ell,m) \geq \pq_{v}(n,n),$$
where
$$n = \begin{cases} 2 \lfloor \frac{\ell-8}{2} \rfloor ,& \text{ when $\ell m$ is even }\\
2 \lfloor \frac{\ell-8}{2}\rfloor -1, & \text{ when $\ell m$ is odd }
\end{cases}$$
and
$v = \min(k, n^2/2)$.
Now apply Corollary~\ref{c:square_bin} to bound $\pq_{v}(n,n)$ and obtain the result for $\ell >16$.

When $\ell \leq 16$, and $m \geq 24$, we can apply the same reasoning as above to show $\pq_k(\ell,m) \geq \pq_{v}(\ell,16)$. Finally, for $\ell,m\leq
16$ the statement is easily verified by direct calculation.
\end{proof}

\begin{proof}[Proof of Theorem~\ref{t:qbin-main}]
For $n\geq \ell-9$, the desired inequality then follows from Theorem~\ref{t:qbin-rect} and the observation that
$$\frac{2^{n/\sqrt{2}} }{ n^{9/2} }  \, \geq  \, 2^{-9/\sqrt{2} } \. \frac{2^{\ell/\sqrt{2}} }{ \ell^{9/2} }\..
$$
Taking \ts $A = 2^{-9/\sqrt{2} }  \ts C \approx 0.004$ \ts gives the desired inequality for all values.
\end{proof}

\smallskip

\subsection{Other bounds}  \label{ss:qbin-upper}
Let $k\le \ell\le m$ and $n=\ell\ts m$. We have:
$$
\pq_k(\ell,m) \. = \. p_k(\ell,m) \ts - \ts p_{k-1}(\ell,m) \.
= \. \parti(k) \ts - \ts \. \parti(k-1) \. \geq \. \parti'(k) \.
\sim \. \frac{\pi}{12\ts\sqrt{2}\ts k^{3/2}} \,\, e^{\pi \sqrt{\frac{2}{3}\ts k}}
$$
Compare this with the lower bound in Theorem~\ref{t:qbin-main}:
$$\pq_k(\ell,m) \, > \, A \, \frac{2^{\sqrt{2\ts k}}}{(2\ts k)^{9/4}}\,.
$$
There is only room to improve the base of exponent here:
$$
\text{from} \ \ 2^{\sqrt{2}} \.\approx \. 2.26 \ \   \text{to}
\ \ e^{\pi \sqrt{\frac{2}{3}}} \. \approx \. 13.00\,.
$$
In fact, using our methods, the best lower bound we can hope
to obtain is
$$
e^{\pi \sqrt{\frac{1}{6}}} \. \approx \. 3.61\,,
$$
which is the base of exponent in the Roth--Szekeres formula
for the number $b_k(n)$ of unrestricted partitions into distinct
odd parts, where~$n\ge k$.

\smallskip

For a different extreme, let $m=\ell$ be even, and $k=m^2/2$.
We have the following sharp upper bound:
$$
\pq_k(m,m) \. \le \. p_k(m,m) \. \sim \. \sqrt{\frac{3}{\pi \ts m}}
\ts  \binom{2\ts m}{m} \. \sim \. \frac{\sqrt{3}\, 4^m}{\pi \ts m}\,.
$$
On the other hand, the lower bound in Theorem~\ref{t:qbin-main} gives:
$$\pq_k(m,m) \, > \, A \, \frac{2^{m}}{m^{9/2}}\,.
$$
Again, we cannot improve the base of the exponent~$2$ with our method,
simply because the total number of partitions into distinct odd
parts~$\le 2m-1$, is equal to~$2^m$.

\bigskip

\section{Effective bounds on Kronecker coefficients for more general triples}\label{sec:gen}

Here we prove a lower bound for a wide class of
partition triples. % as stated in Theorem~\ref{thm:gen}.

Let $\mu$ be a partition with Durfee square of size~$\geq n$. We say that $\mu$ is
\emph{decomposed} as $(n^n+\al,\be)$ if $\al,\be $ are the partitions which
remain after an $n \times n$ square is removed from~$\la$: $\al$ is the remaining partition occupying the first $n$ rows and $\be$ is the partition below, i.e. $\be = (\la_{n+1},\la_{n+2},\ldots)$.

\begin{thm}\label{thm:gen}
Let $\la, \mu, \nu$ be partitions of the same size, such that the Durfee squares of $\mu$ and $\nu$ have sizes at least $n$, and such that the Young diagrams of $\la, \mu, \nu$ can be decomposed as
$\mu = (n^n + \al^1, \al^2 ), \nu = (n^n + \be^1, \be^2)$ and $\la = \gamma^1+\gamma^2 + \tau$, where $\tau = (n^2-k,k)$ or $\tau=(n^2-k,k)'$ for some $k>1$. Suppose that
$g( \al^1, \be^1, \ga^1)>0$ and $g( \al^2, \be^2,   \ga^2)>0$. Then
$$g(\la,\mu,\nu)\,  \geq\,  C  \, \frac{2^{\sqrt{2k}}}{(2k)^{9/4}}\,,
\quad \text{where} \quad C = \frac{\sqrt{27/8}}{\pi^2}\,.
$$
\end{thm}

\begin{proof}
% The semigroup property for the Kronecker coefficients, originally attributed to Brion and reproved in~\cite{CHM,Man}, states that
% if there are two triples of partitions $\la^1,\mu^1,\nu^1$ and $\la^2, \mu^2,\nu^2$, such that $g(\la^1,\mu^1,\nu^1)>0$ and
% $g(\la^2,\mu^2,\nu^2)>0$, then $g(\la^1+\la^2, \mu^1+\mu^2, \nu^1+\nu^2)\geq \max( g(\la^1,\mu^1,\nu^1), g(\la^2, \mu^2,\nu^2))$.
%
Since transposing two partitions in a triple leaves the Kronecker unchanged, by Corollary~\ref{c:square_bin} we have:
\begin{equation}\label{eq:square1}
g(n^n, n^n, \tau) = g(n^n, (n^n)' , (n^2-k,k)') = g(n^n, n^n, (n^2-k,k) )  \, \geq \, C \frac{2^{\sqrt{2k} } }{(2k)^{9/4} }\ts,
\end{equation}
so in particular $g(n^n, n^n, \tau)>0$ with the above effective lower bound. Now we apply the transposition property
and the monotonicity property several times as follows:
$$
\aligned
g( \la, \mu, \nu) \, & = \. g(\la, \mu',\nu') = g\bigl(\tau+\gamma^1+\gamma^2, (n^n+\al^1)' + (\al^2)' , (n^n +\be^1)' +(\be^2)' \bigr)
\\
& \geq \, \max \bigl\{ g \bigl( \tau + \gamma^1, (n^n +\al^1)', (n^n+\be^1)'\bigr), \. g\bigl( \gamma^2, (\al^2)', (\be^2)'\bigr)  \bigr\} \\
& \ge \,
\max  \bigl\{ g( \tau + \gamma^1, n^n +\al^1, n^n+\be^1), g( \gamma^2, \al^2, \be^2)  \bigr\} \\
& \geq \, g ( \tau + \gamma^1, n^n +\al^1, n^n+\be^1)
\\
& \geq \, \max \{ g(\tau, n^n, n^n) , g(\gamma^1,\al^1,\be^1) \} \, \geq \, g(\tau, n^n, n^n)\ts.
\endaligned
$$
The lower bound now follows from equation~\eqref{eq:square1}.
\end{proof}

\begin{proof}[Proof of Corollary~\ref{cor:general}]
Let~$\la$ be decomposed as $\la = (m^m + \al, \be)$, where $\al\vdash a$, $\be \vdash b$.
Since
$$g(\al, \al, (a) ) \. = \. g(\be, \be, (b) )\. = \. 1 \. >0\.,
$$
we can apply Theorem~\ref{thm:gen} for the triple $(\la, \la, \nu)$,
where $\nu = (m^2 -k,k) + (a) + (b) = (n-k,k)$.
\end{proof}

\bigskip

\section{Final remarks} \label{s:fin}

% \subsection{} \label{ss:fin-pre}
% The results of this paper originally appeared in preprint~\cite{PP_bo},
% which also contained other results.  We then split~\cite{PP_bo} into
% two parts, moving the lower bounds into this paper, and various (known and new)
% upper bounds are shifted to a survey~\cite{PP-future}.  We also added new
% applications in Section~\ref{s:char}.

\subsection{} \label{ss:fin-GCT}
It is rather easy to justify the importance of the Kronecker coefficients
in Combinatorics and Representation Theory.  Stanley writes:
``One of the main problems in the combinatorial representation theory
of the symmetric group is to obtain a combinatorial interpretation for
the Kronecker coefficients''~\cite{Sta}.

Geometric Complexity Theory (GCT) is a more recent interdisciplinary
area, where computing the Kronecker coefficients is crucial (see~\cite{MS}).
B\"urgisser voices a common complaint of the experts:
``frustratingly little is known about them''~\cite{Bur}.
%We refer to~\cite{PP_c,PP-future} for details and further references.

Part of this work is motivated by questions in GCT.  Specifically, the interest is
 in estimating the coefficients
$$
g\bigl(m^\ell,m^\ell, \la)\ts, \quad \text{where} \ \ \ \la \vdash \ell\ts m\ts.
$$
Both theorems~\ref{t:char_effective} and~\ref{t:qbin-main}
are directly applicable to this case, when $m=\ell$ and $\la=\la'$,
and when $\ell(\la)=2$, respectively.
We plan to return to this problem in the future.

\subsection{}  \label{ss:fin-border}
The notion of $s$-border equal partitions given in $\S$\ref{ss:border}
is perhaps new but rather natural. It would be interesting to see if
a stronger bound
$$
g(\la,\mu,\mu) \, \ge \,g\bigl(\la^\bos,\mu^\bos,\mu^\bos\bigr)
$$
holds under the assumptions of Corollary~\ref{cor:border}.
Note that when one of the partitions of decomposition is empty,
$\la = (m^m+\al,\varnothing)$, $m = d(\la)$, the bound follows
immediately from Theorem~\ref{thm:gen}.  Note also
that one cannot drop the $\mu=\mu'$ assumption here.  For example, when
$\la=\mu=(2,2,1)$ and $s=1$, we have  $\la^\boss=\mu^\boss=(1)$, and
$g(\la,\mu,\mu)=0$ while $g(\la^\boss,\mu^\boss,\mu^\boss)=1$.

\subsection{}  \label{ss:fin-char}
In a special case $\la=\mu=\mu'$, Theorem~\ref{t:char_effective} and
the Murnaghan--Nakayama rule gives a weak bound $g(\mu,\mu,\mu) \ge 1$
proved earlier in~\cite{BB}.  In~\cite{PPV}, we apply a qualitative
version of Theorem~\ref{t:char_effective} to a variety of partitions generalizing
hooks and two-row partitions.
Unfortunately, computing the characters of $S_n$ is
\SP-hard~\cite{PP_c}.  It is thus unlikely that this
approach can give good bounds for general
\emph{tensor squares} $g(\la,\mu,\mu)$ (cf.~\cite{Val2}).

\subsection{}  \label{ss:fin-asym}
As we showed in $\S$~\ref{ss:qbin-upper}, there is a gap in the base
of the exponent between lower and upper bounds even for $\de_k(m,m)$.
Finding sharper lower bounds in this case would be very interesting.

On the other hand, there is perhaps also room for improvement for
the rectangular case $\de_k(\ell,m)$, where $\ell = o(m)$ and $k$ is
in the ``middle range'' $\ell^2/2 \le k \le \ell \ts m/2$.  Since our
proof of a lower bound in this case relies crucially on
Theorem~\ref{t:manivel} and no other tools in this case are available,
extending the inequality in the theorem would be very useful.

\subsection{}  \label{ss:fin-qbin}
As we mentioned in the introduction, the first lower bound
$\pq_k(\ell,m)\ge 1$ was obtained by the authors in~\cite{PP_s}.
This was quickly extended in followup papers~\cite{Dha} and~\cite{Zan},
both of which employed O'Hara's combinatorial proof~\cite{O}.  First,
Zanello's proof gives:
$$
\pq_k(\ell,m) \. > \. d\., \ \ \ \text{for}
\ \ \ell \. \ge \. d^2+5d+12\ts, \ \, m \. \ge \. 2d+4\ts, \ \,
4d^2 + 10d+7 \le \. k \. \le \ell m/2\ts,
$$
see the proof of Prop.~4 in~\cite{Zan}.  Similarly, Dhand proves
somewhat stronger bounds
$$
\pq_k(\ell,m) \. > \. d\., \ \ \ \text{for} \ \
\ell, \ts m \. \ge \. 8\ts d, \ \ 2 d \le \. k \. \le \ell m/2\ts,
$$
see Theorem~1.1 in~\cite{Dha}.

In conclusion, let us mention that the sequence $p_k(\ell,m)$ has
remarkably sharp asymptotics bounds,
including the \emph{central limit theorem} (CLT)
with the error bound~\cite{Tak}, and
the Hardy--Ramanujan type formula~\cite{AA}.  When
$\ell$ is fixed, sharp asymptotic bounds are given in~\cite{SZ}.

\vskip.3cm

\noindent
{\bf Acknowledgements.} \ We are grateful to Sasha Barvinok,
Jes{\'u}s De Loera, Stephen DeSalvo, Richard Stanley and Fabrizio Zanello
for many helpful conversations.  Special thanks to Ernesto Vallejo for
reading the early draft of the paper and help with the  references.
This work was done while the first author was partially supported by the NSF grants,
and the second by the Simons Postdoctoral Fellowship while a postdoc at UCLA.

\vskip.8cm

% \newpage

%
%%%%%%%%%%%%%%%%%%%%%%%%%%%%%%%%%%%%%%%%%%%%%%%%%%%%%%%%%%%%%%%%%%%%%%%%

%%%%%%%%%%%%%%%%%%%%%%%%%%%%%%%%%%%%%%%%%%%%%%%%%%%%%%%%%%%%%%%%%%%%%%%%
\end{document}